\documentclass[pdflatex,sn-mathphys]{sn-jnl}% Math and Physical Sciences 
\jyear{2021}%

\theoremstyle{thmstyleone}%
\newtheorem{theorem}{Theorem}%  meant for continuous numbers
\theoremstyle{thmstyletwo}%
\newtheorem{remark}{Remark}%

\theoremstyle{thmstylethree}%

\usepackage[leqno]{mathtools}
\usepackage{enumerate}
\usepackage[mathscr]{eucal}
\normalbaroutside
\usepackage{tikz}
\usepackage[center]{caption}

\usepackage{amsmath}
\usepackage{amssymb}
\usepackage{amsthm}

\usepackage{mathtools}
\usepackage{stmaryrd}
\usepackage{tikz-cd}

\theoremstyle{plain}

\newtheorem{cor}{Corollary}[theorem]
\newtheorem{lemma}{Lemma}[section]

\theoremstyle{definition}

\begin{document}

\title[]{Subgroup perfect codes of $S_n$ in Cayley sum graphs}

%%=============================================================%%
%% Prefix	-> \pfx{Dr}
%% GivenName	-> \fnm{Joergen W.}
%% Particle	-> \spfx{van der} -> surname prefix
%% FamilyName	-> \sur{Ploeg}
%% Suffix	-> \sfx{IV}
%% NatureName	-> \tanm{Poet Laureate} -> Title after name
%% Degrees	-> \dgr{MSc, PhD}
%% \author*[1,2]{\pfx{Dr} \fnm{Joergen W.} \spfx{van der} \sur{Ploeg} \sfx{IV} \tanm{Poet Laureate} 
%%                 \dgr{MSc, PhD}}\email{iauthor@gmail.com}
%%=============================================================%%

\author[]{\fnm{Ankan} \sur{Shaw}}\email{ankanf22@gmail.com}
\equalcont{These authors contributed equally to this work.}

\author[]{\fnm{Biswajit} \sur{Mondal}}\email{mbabay2018@gmail.com}
\equalcont{These authors contributed equally to this work.}

\author[]{\fnm{Satya} \sur{Bagchi}}\email{sbagchi.maths@nitdgp.ac.in}
\equalcont{These authors contributed equally to this work.}
%\author*[]{\fnm{Satya} \sur{Bagchi}}\email{sbagchi.maths@nitdgp.ac.in}

%\equalcont{These authors contributed equally to this work.}
%\author[1,2]{\fnm{Third} \sur{Author}}\email{iiiauthor@gmail.com}
\equalcont{These authors contributed equally to this work.}

\affil[]{\orgdiv{Department of Mathematics}, \orgname{National Institute of Technology Durgapur}, \orgaddress{\city{Durgapur}, \postcode{713209}, \state{West Bengal}, \country{INDIA}}}

%\affil[2]{\orgdiv{Department of Mathematics}, \orgname{National Institute of Technology Durgapur}, \orgaddress{\street{Street}, \city{Durgapur}, \postcode{713209}, \state{West Bengal}, \country{INDIA}}}

%\affil[3]{\orgdiv{Department}, \orgname{Organization}, \orgaddress{\street{Street}, \city{City}, \postcode{610101}, \state{State}, \country{Country}}}

%%==================================%%
%% sample for unstructured abstract %%
%%==================================%%

\abstract{A perfect code in a graph $\Gamma = (V, E)$ is a subset $C$ of $V$ such that no two vertices in $C$ are adjacent, and every vertex in $V \setminus C$ is adjacent to exactly one vertex in $C$. Let $ G $ be a finite group, and let $ S $ be a square-free normal subset of $ G $. The Cayley sum graph of $ G $ with respect to $ S $ is a simple graph with vertex set $ G $ and two vertices $ x $ and $ y $ are adjacent if $ xy\in S .$ A subset $ C $ of $ G $ is called perfect code of $ G $ if there exists a Cayley sum graph of $ G $ that admits $ C $ as a perfect code. In particular, if a subgroup of $ G $ is a perfect code of $ G $, then the subgroup is called a subgroup perfect code of $ G $. In this work, we prove that there does not exist any proper perfect subgroup code of symmetric group $ S_n $. Using this result, we provide a complete characterization of the perfect subgroup code of the alternating group $A_n$.}

%%================================%%
%% Sample for structured abstract %%
%%================================%%

%\subjclass[2020]{Primary 94B05, 94B15, 20C05}
%%\pacs[JEL Classification]{D8, H51}


\maketitle

\section{Introduction}

Every group and every graph considered in the study are finite. In this work, we use conventional group-theoretic vocabulary and notation, which are available, for instance, in \cite{Huang2018, Suzuki19821}.

Let $\Gamma$ be a graph with edge set $E(\Gamma)$ and vertex set  $V(\Gamma)$. If there is no path in $\Gamma$ connecting two vertices, the distance in $\Gamma$ between them is equal to $ \infty$. Otherwise, it equals the shortest path length between the two vertices. 
Let $ t $ be a natural number and $ C $ be a subset of $V(\Gamma)$. If every vertex of $ \Gamma $ is at a distance no greater than $t$ from precisely one vertex in $C$, then $ C $ \cite{Chen2020,Krotov2020} is said to be a perfect $t$-code in $ \Gamma $. A perfect code is generally a perfect $1$-code. In equivalent terms, if $C$ is an independent set of $\Gamma$ and each vertex in $V(\Gamma) \setminus C$ is adjacent to precisely one vertex in $C$, then subset $C$ of $V(\Gamma)$ is a perfect code in $\Gamma$. A perfect code in a graph is also referred to as an efficient dominating set \cite{Dejter2003, Lee2001} of the graph.

Let $G$ be a finite group and $S$ be a normal subset of $G$ (i.e., $g^{-1}Sg = S$ for all $g \in G$). Note that in an abelian group, every subset is automatically normal. An element $g \in G$ is called a \emph{square} if there exists $h \in G$ such that $g = h^2$. A subset of $G$ is said to be \emph{square-free} if it contains no square elements.

Given a subset $S \subseteq G$, the \emph{Cayley sum graph} $CayS(G, S)$ is defined as the directed graph whose vertex set is $G$, with two vertices $x$ and $y$ connected by an edge if and only if $xy \in S$. When $S$ is normal, the graph becomes undirected since the condition $xy \in S$ is equivalent to $yx \in S$. Furthermore, if there exists $g \in G$ such that $g^2 \in S$, then $\{g,g\}$ forms a \emph{semi-edge} (an edge with a single endpoint). The graph is called \emph{simple} if it contains neither multiple edges nor semi-edges.

Initially, Biggs \cite{Biggs1973}  made known the concept of perfect $t$-codes in a graph. He defined perfect $t$-codes under the Hamming and Lee metrics as a generalization. Perfect codes in Cayley graphs have been extensively studied; see \cite{Huang2018} for a survey of early results and \cite{Ma, Zhang2021, Zhang2022, Zhang2023, Zhang24} for recent advances in this area. In \cite{MacWilliams1977}, the authors showed that the Hamming distance of two code words of length $n$ over an alphabet of size $ m > 1 $ corresponds precisely to the graph distance in the Hamming graph $ H(n, m) $. Thus, under the Hamming metric, perfect $t$-codes in $H(n, m)$ are precisely those defined in the classical context. Similarly to this, the graph distance in the Cartesian product $L(n, m)$ of $n$ copies of the cycle of length $m$ equals the Lee distance \cite{Horak} between words of length $n$ across an alphabet of size $m > 2$.

Note that all graphs considered in this paper are simple. Thus, we always consider a simple Cayley sum graph $CayS(G, S)$, that is, the connecting set $S$ needs to be normal and square-free. More explicitly, for a square-free normal subset $S$ of $G$, the Cayley sum graph $CayS(G, S)$ of $G$ with respect to the connecting set $S$ is a simple graph with vertex set $G$ and two vertices $x$ and $y$ are adjacent if $xy\in S$.

The notion of Cayley sum graphs was initially introduced for abelian groups before being extended to arbitrary groups in \cite{Amo}. In $1989$, Chung \cite{Chung} first introduced the Cayley sum graphs of abelian groups. As pointed out in \cite{Gry}, the twins of the usual Cayley graphs, Cayley sum graphs are rather difficult to study, so that they received much less attention in the literature. Most results on Cayley sum graphs can be found in \cite{Amo, Che, ev, Kon, Lev, Ma2016}. In \cite{Ma2020}, authors studied perfect codes of Cayley sum graphs and define a subgroup perfect
code of a group by using Cayley sum graphs instead of Cayley graphs. More precisely, a
subgroup perfect code of a group G is a subgroup of G and a perfect code of some Cayley
sum graph of G. They gave a necessary and sufficient condition for a non-trivial subgroup of
an abelian group with non-trivial Sylow $2$-subgroup to be a subgroup perfect code of the
group. In very recently, J Zhang \cite{Zhang2024} gave a necessary conditions of a subgroup of a given group being a (total) perfect code in a Cayley sum graph of the group. In there, the Cayley sum graphs of some families of groups which admit a subgroup as a (total) perfect code are classified.

In this paper, we extend this study to the symmetric group $ S_n $, examining subgroup perfect codes in its Cayley sum graph. As a key application of our results, we further explore perfect subgroup codes in the alternating group $ A_n $. Our work provides a complete classification of subgroup perfect codes for both  $S_n$ and $ A_n $ in the Cayley sum graph setting, contributing new structural insights into coding theory in non-abelian groups.

Excluding the introductory part, the present work is demarcated into two sections.  Section \ref{sec2} lays the groundwork by presenting essential definitions, established results, and notations that form the basis of our study. Section \ref{sec3} contains our principal contributions, offering a complete classification of subgroup perfect codes of the symmetric group $S_n$ in the Cayley sum graph. Furthermore, we extend these results to examine subgroup codes of the alternating group $A_n$ in the Cayley sum graph, demonstrating how the structural properties uncovered for $ S_n $ naturally apply to its subgroup $A_n$ in the Cayley sum graph.

\section{Preliminaries}\label{sec2}
In this section, we introduce some definitions and preliminary results for the reader's convenience. Let $S_n$  and $ A_n $ denote the symmetric group  and alternating group on $ T = \{1,\,2,\dots,\,\,n\}$ with identity element $1_{S_n}$, where $ n>2$. For a subset $S \subseteq S_n$, we write $|S|$ for its cardinality and $\langle S \rangle$ for the subgroup generated by $S$. The order of a permutation $\sigma \in S_n$ is denoted by $o(\sigma)$, and $\text{cl}(\sigma)$ represents its conjugacy class. Let $ \mathcal{N}(\sigma) $ denote the number of disjoint transpositions in the cycle decomposition of $ \sigma.$ A permutation $\sigma$ is called an involution if $o(\sigma) = 2$. Let us define $ \Delta(\sigma) = \{ i\in T:~ \sigma(i)\neq i\} $, and $ I(H) = \{\sigma\in H: o(\sigma) = 2\} $ for any subgroup $ H $ of $ S_n $. Let $ t = (1\,\,2)(3\,\,4)\cdots(2k-1\,\,2k)$ be a product of $ k $ disjoint transpositions in $S_n$. For any subgroup $H$ of $S_n$, we define the following subsets:
\begin{align*}
     I_k(H) = cl(t)\cap H,\,\, \text{and}\,\,  I_{min}(H) =\min \{ k\in\mathbb{N}\cup \{0\}: I_k(H) \neq \emptyset \}.  
\end{align*}  

\noindent An element  $ \sigma \in S_n $ acts nontrivially on a transposition $ (a\,\,b)\in S_n $ if at least one of the following holds:
\begin{enumerate}[(i)]
    \item $ \sigma (a) \neq a $,
    \item $ \sigma (b) \neq b $. 
\end{enumerate}
More generally, if $ \Delta(\sigma) \cap \Delta (x) \neq \emptyset $, where $ x\in S_n$, then $\sigma$ acts nontrivially on $x$.
 
\begin{lemma}\label{kc1}\cite{Cornad}
For $n > 2$, $S_n$ is generated by the $n-1$ transpositions $ (1\,\,2), (2\,\,3), \dots, (n-1\,\, n)$.
\end{lemma}

\begin{lemma}\label{kc2}\cite{Cornad}
For $n > 2$, $S_n$ is generated by the $n-1$ transpositions $ (1\,\,2), (1\,\,3), \dots, (1\,\, n)$.
\end{lemma}

\begin{lemma}\label{kc3}\cite{Cornad}
For $1 \leq a < b \leq n$, the transposition $(a\,\,b)$ and $n$-cycle $(1\,\,2\dots n)$ generate $S_n$ if and only if $(b-a, n) = 1$.
\end{lemma}
\begin{lemma}\cite{Cornad}
The group $ A_n $ is simple for $ n>5 $.
\end{lemma}

\begin{lemma}\cite{Cornad}
For $n > 2$, each element of $A_n$ is a product of $3$-cycles. Therefore the $3$-cycles
generate $A_n$.
\end{lemma}

\begin{lemma}\cite{Cornad}\label{lmn2.6}
Suppose $ H = \langle x\in A_n: o(x) = 2 \rangle $ is subgroup of $ S_n $. Then $ H = A_n.$ 
\end{lemma}

\begin{lemma}\label{lmn2.8}\cite{Dixon}
Suppose $ H = \{g^2: g\in S_n\}$ is the set of all square elements of $ S_n $. Then $ H = A_n $ for $ n\leq 5 $ otherwise, $ H\subset A_n $ such that $ H $ is not subgroup of $ S_n$ and $ A_n\setminus H \neq \emptyset$
\end{lemma} 

\begin{lemma}\label{lmn2.9}\cite{malik97}
Suppose $ x = (i_1\,\,i_2\,\,\dots i_m) \cdots (j_1\,\,j_2\,\,\dots j_l) \in S_n $. Then for all $ y \in S_n $,  $ yxy^{-1} = (y(i_1)\,\,y(i_2)\dots y(i_m)) \cdots (y(j_1)\,\,y(j_2)\dots y(j_l))$.
\end{lemma} 
\begin{lemma}\label{lm 1}
 For $ x,\, a \in S_n $ such that $ x\neq a^{-1}$ and $ o(x) = 2.$ Then $ cl(xa) = cl(xa^{-1}) = cl(ax) $.   
\end{lemma}
 \begin{proof}
Since $ xa = a^{-1}(ax)a $ and $ ax = x^{-1}(xa)x$, therefore, $ cl(xa) = cl(ax).$ We know that $ cl(y) = cl(y^{-1})$ for any $y\in S_n$. Then $ cl(ax)=cl((ax)^{-1}) = cl (xa^{-1})$ as $ x = x^{-1}$ and this completes the proof. 
 \end{proof}   
\begin{lemma}\label{lmn2.10}
Let $ x = (i_1\,\,i_2)\cdots (i_{2r-1}\,\,i_{2r})$ and $ y =(t_1\,\,t_2)\cdots(t_{2l-1}\,\,t_{2l})$ be two involutions in $ S_n$. Then $ xy = yx $ if and only if there exist $ k_1, k_2, k_3 \in \{0, 1, \dots, r\} $ such that $ (t_1\,\,t_2)\cdots(t_{2k_1-1}\,\,t_{2k_1}) = (i_1\,\,i_2)\cdots(i_{2k_1-1}\,\,i_{2k_1})$, $ (t_{2k_1+1}\,\,t_{2k_1+2})(t_{2k_1+3}\,\,t_{2k_1+4}) \cdots  (t_{2m_1-3}\,\,t_{2m_1-2})  (t_{2m_1-1}\,\,t_{2m_1}) = (i_{2k_1+1}\,\,i_{2k_1+3}) (i_{2k_1+2}\,\,i_{2k_1+4})\cdots(i_{2m_1-3}\,\,i_{2m_1-1})(i_{2m_1-2}\,\, i_{2m_1})$, $ (t_{2m_1 + 1}\,\,t_{2m_1+2})(t_{2m_1 + 3}\,\,t_{2m_1 + 4}) \cdots  (t_{2m_2-3}\,\,t_{2m_2-2})  (t_{2m_2-1}\,\,t_{2m_2}) = (i_{2m_1+1}\,\,i_{2m_1+4}) (i_{2m_1+2}\,\,i_{2m_1+3})\cdots(i_{2m_2-2}\,\,i_{2m_2-1})(i_{2m_2-3}\,\, i_{2m_2})$ and $ \Delta\left( (t_{2m_2 + 1}\,\,t_{2m_2 + 2})\cdots(t_{2l-1}\,\,t_{2l}) \right) \cap \Delta (x) = \emptyset$, where $m_1 = k_1 + k_2$, $m_2 = k_1 + k_2 + k_3 $; and $ k_2$, $k_3$ are even.
\end{lemma}

\begin{proof}
    Let $ x = (i_1\,\,i_2)\cdots(i_{2r-1}\,\,i_{2r}) $ and $ y = (t_1\,\,t_2)\cdots(t_{2l-1}\,\,t_{2l})$ be two involutions of $ S_n $ such that $ yxy^{-1} = x.$ Then It follows from Lemma \ref{lmn2.9} that $ (y(i_1)\,\,y(i_2))\cdots(y(i_{2r-1})\,\,y(i_{2r})) = (i_1\,\,i_2)\cdots(i_{2r-1}\,\,i_{2r}).$ To proceed further, we consider the following cases.

\noindent \textbf{Case I:} When $ (y(i_j)\,\,y(i_{j+1})) = (i_j\,\,i_{j+1})$ for some $ j\in \{1,\dots,~2r-1\} $.

\noindent Without loss of generality, let $ j = 1$.

\noindent Suppose $ y(i_1) = i_1 $. Then $ y(i_2) = i_2.$ This shows that $ \{i_1,\,\,i_2\}\cap \Delta(y) = \emptyset.$ Consequently, $ (i_1\,\,i_2) $ and $ y $ are disjoint.

\noindent Suppose $ y(i_1) = i_2 $, then $ (i_1\,\,i_2) = (t_{k}\,\,t_{k+1}) $ for some $ k\in \{1,\dots,~2l-1\}$ as $ o(y) = 2 $. This shows that $ x $ and $ y $ share a transposition.

\noindent \textbf{Case II:} When $ (y(i_k)\,\,y(i_{k+1})) = (i_m\,\,i_{m+1})$ for some $ k,\,\,m \in \{1,\dots,~2r-1\} $ such that $ k\neq m$.

\noindent Without loss of generality, let $ k = 1,~ m = 3$. Suppose $ y(i_1) = i_3 $, then $ y(i_2) = i_4$. This shows that $ y(i_3) = i_1 $ and $ y(i_4) = i_2 $ as $ y $ is an involution. Therefore, $(i_1\,\, i_3) = (t_a\,\,t_{a+1})$ and $(i_2\,\,i_4) = ( t_b\,\,t_{b+1}) $ for some $ a \neq b \in \{1,~2,\dots,~2l-1\}$. Suppose $ y(i_1) = i_4 $, then $ y(i_2) = i_3$. Then we proceed with similar way, we get a pair of transpositions $ (T_1,\,\,T_2)$ of $ x$ and a pair of transpositions $ (U_1,\,\,U_2)$ of $ y $ such that $ \Delta(T_1T_2) = \Delta(U_1U_2)$, and this concludes this case.

\noindent The converse is straightforward to establish.
\end{proof}

\begin{remark}\label{rmk01}
  The above lemma demonstrates that if $|\Delta(x)\cap \Delta (y)| $ is odd for $ x,\,\,y \in I(S_n)$, then $ xy \neq yx.$    
\end{remark}

\section{Perfect codes}\label{sec3}
In this section, we study perfect codes of $ S_n $ in the Cayley sum graph. Using those results, we characterize the perfect codes of $ A_n $ in the Cayley sum graph. These results give the complete characterizations of perfect codes of $ S_n $ and $ A_n $ in the Cayley sum graph.
\begin{lemma}\cite{Ma2020}\label{lmn3.1}
Let $S$ and $H$ be a square-free normal subset and a subgroup of $G$, respectively.
The following are equivalent.
\begin{enumerate}[(i)]
    \item $H$ is a perfect code of $CayS(G, S)$.
    \item $ S\cup\{1_G\} $ is a right transversal of $H$ in $G$.
    \item $|G : H| =  |S| + 1$ and $H \cap(S\cup SS^{-1}) = \{1_G\}$.
\end{enumerate}
\end{lemma}
% The structure of the intersections between conjugacy classes and cosets of $H$ plays a crucial role in determining whether $H$ can form a perfect subgroup code. The following result is the counterpart of Lemma \ref{lmn3.1}, which formalizes a necessary condition for $H$ to be not perfect.

\begin{lemma}\label{lmn3.2}
A subgroup $H$ of $G$ is not a perfect subgroup code of $G$ if there exists an element $x \in G \setminus H$ such that the coset $xH$ satisfies either of the following conditions:

\begin{enumerate}[(i)]
\item every element of $xH$ is square in $G$; or
\item for every non-square element $z \in xH$, there exists $w \in G$ such that $|cl(z) \cap wH| > 1$ and/or $cl(z) \cap H \neq \emptyset$.
\end{enumerate}
\end{lemma}

\begin{proof}
    This follows immediately from Lemma \ref{lmn3.1}.
\end{proof}

\begin{theorem}\label{thm002}\label{lmn3.3}
Let $ H $ be a nontrivial odd-order subgroup of $ S_n$. Then $ H $ is not a perfect subgroup code of $ S_n $ in the Cayley sum graph.
\end{theorem}

\begin{proof}
Let $ H $ be an odd-order subgroup of $ S_n$ such that $ H = \{ 1_{S_n}, y_1, y_1^{-1}, \dots, y_m, y_m^{-1}\}.$ It is clear that $ H \subset A_n $ as any odd-order element in $ S_n $ is even permutation. Since any odd-order element of $ S_n $ is square element, therefore, every element of $ H $ is a square element in $S_n$. To proceed further, we consider the following two cases.
 
\noindent\textbf{Case I:} When $ n\leq 5.$

\noindent Then $ S_n^2 = A_n$ due to Lemma \ref{lmn2.8}. If $ A_n\setminus H \neq \emptyset $, then there exists $ x\in A_n\setminus H $ such that $ xH \subset A_n = S_n^2.$ Therefore, $ xH $ contains all square elements. Then it follows from Lemma \ref{lmn3.2} that $ H $ is not perfect. If $ H = A_n $, then $ S_n  $ as $ |cl(x)\cap (S_n\setminus H)| \neq 1$ for $ x \in S_n\setminus H$. Therefore, it follows from Lemma \ref{lmn3.2} that $ H $ is not perfect.

\noindent\textbf{Case II:} When $ n>5$. 

\noindent Then $ A_n\setminus S_n^2\neq \emptyset$ due to Lemma \ref{lmn2.8}. Since every element of $ H $ is a square element in $S_n$, therefore, $ H \subset S_n^2$. If follows from Lemma \ref{lmn2.8} that $ S_n^2 $ is not a subgroup, therefore, $ S_n^2 \setminus H \neq \emptyset.$ Let us consider an involution $ x \in S_n^2 \setminus H $ such that $x = (1\,\,2)(3\,\,4).$ Since $ H $ is of odd-order subgroup, therefore, every nonidentity element in $ H $ forms a distinct pair with its inverse. Thus, $ xH = \{ x,\,\, xy_1,\,\, xy_1^{-1}, \dots, \,\,xy_m, \,\,xy_m^{-1}\}.$ It follows from Lemma \ref{lm 1} that $ |xH \cap cl(\sigma) | \neq 1 $ for non-square $ \sigma \in S_n\setminus H $. Therefore, it follows from Lemma \ref{lmn3.2} that $ H $ is not perfect of $ S_n $ in the Cayley sum graph. 
\end{proof}

Note that one can take $ x $ as a product of any even number of disjoint transpositions in $ S_n$. It is easy to see that any odd-order subgroup of $ S_n$ is also a subgroup of $ A_n $. Thus, the following result follows from Theorem \ref{eqn02} immediately.
\begin{cor}\label{col1 }
Let $ H $ be a non-trivial odd-order subgroup of $ A_n$. Then $ H $ is not a perfect subgroup code of $ A_n $ in the Cayley sum graph.
\end{cor}

\begin{theorem}
Let $ H $ be a nontrivial even order subgroup of $ S_n $ such that $ I_{min}(H) = 1$. Then $ H $ is not a perfect subgroup code of $ S_n $ in the Cayley sum graph.
\end{theorem}
\begin{proof}
\noindent Without loss of generality, let $ x = (1\,\,2)\in H $, $ y\in I_r(H) $, $ \Delta(I(H)) = \{1,~2,\dots,~l\}$ and $ J(H) = \{ 1_{S_n}, y_1, y_1^{-1}, \dots, y_m, y_m^{-1}\}.$ Since $ H \neq S_n$, then it follows from Lemmas \ref{kc1}, \ref{kc2} that there exists $ (a\,\,b) \in S_n $ such that $ (a\,\,b)\notin H $, where $ a\in \Delta(H)$. Without loss of generality, let $ z = (a\,\,b) = (1\,\,n).$ It follows from Lemma \ref{lm 1} that $| cl(zy_i)\cap zH|>1$ for $ 1\leq i \leq m $ and $ cl(z)\cap H \neq\emptyset.$ If $ |I(H)| = 1 $, then $cl(z)\cap H \neq \emptyset $ as $ cl(z) = cl(x) .$  Consequently, we are done due to Lemma \ref{lmn3.2}.
\vspace{0.1cm}

\noindent  Suppose $ |I(H)|>1 $. We claim that for every non square element $\sigma \in zH $, there exists $ w \in S_n $ such that $ |cl(\sigma)\cap wH|>1 $ or $ cl(\sigma)\cap H \neq \emptyset $ or both. To proceed further, we consider the following cases.

\vspace{0.1cm}

\noindent\textbf{Case I:} When $ |\Delta(z)\cap\Delta(y)| = 0 .$

\noindent Let $ y = (i_1\,\,i_2)\cdots(i_{2r-1}\,\,i_{2r}) $. Then $ zy = (1\,\,n)(i_1\,\,i_2)\cdots(i_{2r-1}\,\,i_{2r}) $. If $ r $ is odd, then $ zy $ is square element. Consequently, we are done. Suppose $ r $ is even. Let us consider the involution $ w = (i_1\,\,i_3)\cdots(i_{2r-2}\,\,i_{2r}) $. Now, $ wz = (1\,\,n)(i_1\,\,i_3)\cdots(i_{2r-2}\,\,i_{2r}) $ and $ wzy = (1\,\,n)(i_1\,\,i_4)(i_2\,\,i_3)\cdots(i_{2r-2}\,\,i_{2r-1})(i_{2r-3}\,\,i_{2r}).$ This shows that $ cl(wz) = cl(wzy)$. Therefore, we are done.
\vspace{0.1cm}

\noindent\textbf{Case II:} When $ |\Delta(z)\cap\Delta(y)| = 1.$

\noindent  Without loss of generality, let $ y = (1 \,\,i_2)(i_3\,\,i_4)\cdots(i_{2r-1}\,\,i_{2r}).$ Then $ zy = (1\,\,i_2\,\,n)(i_3\,\,i_4)\cdots(i_{2r-1}\,\,i_{2r}).$ If $ r $ is odd, then $ zy $ is a square element. Consequently, we are done. Suppose $ r $ is even. Let us consider the involution $ w = (1\,\,i_4)(i_2\,\,i_3)w_1$, where $w_1 = (i_5\,\,i_7)\cdots(i_{2r-2}\,\,i_{2r}).$ Then $ wz = (1\,\,n\,\,i_4)(i_2\,\,i_3) w_1 $ and $ w(zy) = (i_2\,\,n\,\,i_4)(1\,\,i_3)(i_5\,\,i_8)(i_6\,\,i_7)\cdots(i_{2r-1}\,\,i_{2r-1})(i_{2r-3}\,\,i_{2r}) $. This shows that $ cl(wz) = cl(wzy) = cl(zy) $. Therefore, we are done.
\vspace{0.1cm}

\noindent\textbf{Case III:} When $ |\Delta(z)\cap\Delta(y)| = 2.$

\noindent Suppose $ z $ and $ y $ share a common transposition. Without loss of generality, let $ y = (1\,\,n)(i_3\,\,i_4)\cdots(i_{2r-1}\,\,i_{2r}).$ Then $ zy = (i_3\,\,i_4)\cdots(i_{2r-1}\,\,i_{2r})$. If $ r $ is odd, then $ zy $ is square element. Consequently, we are done. Suppose $ r $ is even. Let us consider the involution $ w = (i_5\,\,i_7)(i_6\,\,i_8)\cdots (i_{2r-2}\,\,i_{2r})$. Then $ wz = (1\,\,n)(i_5\,\,i_7)(i_6\,\,i_8)\cdots (i_{2r-2}\,\,i_{2r})$ and $ w(zy) = (i_3\,\,i_4)(i_5\,\,i_8)(i_6\,\,i_7)\cdots(i_{2r-2}\,\,i_{2r-1})(i_{2r-3}\,\,i_{2r}).$ This shows that $ cl(wz) = cl(wzy) = cl(zy)$, Therefore, we are done.

\noindent Suppose $ z $ and $ y $ do not share a common transposition. Without loss of generality, let $ y = (1\,\,i_2)(n\,\,i_4)(i_5\,\,i_6)\cdots(i_{2r-1}\,\,i_{2r}).$ To proceed further, we consider the following two subcases.
 \vspace{0.1cm}
 
\noindent\textbf{Subcase A:} When $ |\Delta(x)\cap\Delta(y)| = 1 $.

\noindent Now, $ zy = (1\,\,i_2\,\,n\,\,i_4)(i_5\,\,i_6)\cdots(i_{2r-1}\,\,i_{2r}).$ It is clear that $ xy \in H $. Then $ (xy)^2 = (1\,\,2\,\,i_2)\in H$ and $ (xy)^{-2} = (1\,\,i_2\,\,2) \in H$ as $ H $ is subgroup of $ S_n.$ Consequently, $ (1\,\,2\,\,i_2)xy = (n\,\,i_4)(i_5\,\,i_6)\cdots(i_{2r-1}\,\,i_{2r})\in H$. Then $ x(1\,\,2\,\,i_2)xy \in H.$ Then $ zx(1\,\,2\,\,i_2)xy = (1\,\,2\,\,n\,\,i_4)(i_5\,\,i_6)\cdots(i_{2r-1}\,\,i_{2r}) \neq zy .$ This shows that $ cl(zy) = cl(zx(1\,\,i_2\,\,2)xy) $. Therefore, we are done.

%\noindent Suppose $ r $ is even. Let us consider the involution $ w = (1\,\,i_4)[(i_5\,\,i_7)\cdots(i_{2r-2}\,\,i_{2r})]$. Then $ w (zx) = (1\,\,2\,\,n\,\,i_4) (i_5\,\,i_7)(i_6\,\,i_8)\cdots(i_{2r-3}\,\,i_{2r-1})(i_{2r-2}\,\,i_{2r})$ and $ w(zxy) = (1\,\,i_2\,\,2\,\,n)(i_5\,\,i_8)(i_6\,\,i_7)\cdots(i_{2r-2}\,\,i_{2r-1})(i_{2r-3}\,\,i_{2r})$. This shows that $ cl(wzx) = cl(wzxy) = cl(zy) $. Consequently, we are done.

\vspace{0.1cm}

\noindent\textbf{Subcase B:} When $ |\Delta(x)\cap\Delta(y)| = 2 $.

\noindent Suppose $ x$ and $ y $ share a transposition. Then $ y = (1\,\,2)(n\,\,i_4)(i_5\,\,i_6)\cdots(i_{2r-1}\,\,i_{2r})$. Now, $ zy = (1 \,\,2\,\,n\,\,i_4)(i_5\,\,i_6)\cdots (i_{2r-1}\,\,i_{2r}).$ It is clear that $ xy = (n\,\,i_4)\cdots(i_{2r-1}\,\,i_{2r})\in H $.

%Suppose $ r $ is even. 

\noindent Now, consider the element $ w = (1\,\,n\,\,2\,\,i_4)[(i_5\,\,i_7)\cdots(i_{2r-2}\,\,i_{2r})].$ Then $ wy = (1\,\,i_4\,\,2\,\,n)(i_5\,\,i_8)(i_6\,\,i_7)\cdots(i_{2r-2}\,\,i_{2r-1})(i_{2r-3}\,\,i_{2r}).$ This shows that $ cl(w) = cl(wy) = cl(zy) $. Since $ w, \,\,wy \in wH $, therefore, we are done.

\vspace{0.1cm}

\noindent Suppose $ x$ and $ y $ do not share a transposition. Without loss of generality, let $ y = (1\,\,i_2)(n\,\,2)(i_5\,\,i_6)\cdots(i_{2r-1}\,\,i_{2r})$. Then $ zy = (1\,\,i_2\,\,n\,\,2)(i_5\,\,i_6)\cdots(i_{2r-1}\,\,i_{2r})$. Now, $ xy = (1\,\,i_2\,\,2\,\,n)(i_5\,\,i_6)\cdots(i_{2r-1}\,\,i_{2r}).$ This shows that $ cl(xy) =  cl(zy) $. Consequently, $ cl(zy)\cap H \neq \emptyset$ and this completes thi case.

\vspace{0.1cm}

\noindent It follows from the above cases that our claim is justified. Hence, every element of $zH$ is square in $S_n$ for some $ z\in S_n\setminus H $; or for every non-square element $\sigma \in zH$, there exists $w\in S_n $ such that $|Cl(\sigma)\cap wH| > 1$ and/or $ cl(\sigma)\cap H \neq\emptyset$.. Consequently, $H$ is not a perfect subgroup code of $S_n$ in the Cayley sum graph due to Lemma \ref{lmn3.2} and this completes the proof.
\end{proof}

\begin{theorem}\label{thm004}
Let $ H $ be a nontrivial even order subgroup of $ S_n $ such that $ I_{min}(H) = 2$. Then $ H $ is not a perfect subgroup code of $ S_n $ in the Cayley sum graph.
\end{theorem}
\begin{proof}
let $ J(H) = \{ 1_{S_n}, y_1, y_1^{-1}, \dots, y_m, y_m^{-1}\}.$ It follows from Lemma \ref{lm 1} that $| cl(zy_i)\cap zH|>1$ for any $ z \in I(S_n)$, for $ 1\leq i \leq m $.

\vspace{0.1cm}

\noindent We claim that for every non square element $\sigma \in zH $ for some $ z\in I(S_n\setminus H)$, there exists $ w \in S_n $ such that $ |cl(\sigma)\cap wH|>1 $ or $ cl(\sigma)\cap H = \emptyset$ or both.

\vspace{0.1cm}

\noindent Now, we divide this proof into two parts.
\vspace{0.1cm}

\noindent\textbf{Case I:} Suppose there exists $ (i\,\,j)(k\,\,l)\in H $ such that $ (i\,\,k)(j\,\,l)\notin H $.

\noindent Without loss of generality, let $ x =(1\,\,2)(3\,\,4)\in H $ and $ z = (1\,\,3)(2\,\,4)\notin H.$ Suppose $ y \in I_r(H)$. If $ y = x $, then $ cl(z) = cl(zy)$. Consequently, we are done as $| cl(zy_i)\cap zH|>1$ and $ cl(z)\cap H \neq \emptyset$.

Suppose $ y\neq x$. It is clear that $ xz= zx.$ To proceed further, we consider following subcases.
\vspace{0.1cm}

\noindent\textbf{Subcase A:} When $ |\Delta(z)\cap\Delta(y)| = 0.$

\noindent Then $cl(zy) = cl(xy)$. Consequently, we are done.
\vspace{0.1cm}

\noindent\textbf{Subcase B:} When $ |\Delta(z)\cap\Delta(y)| = 1.$ 

\noindent Without loss of generality, let $ y = (1\,\,i_2)\cdots(i_{2r-1}\,\,i_{2r})$. Now, $ x(1\,\,i_2)x = (1\,\,2)(3\,\,4)(1\,\,i_2)(1\,\,2)(3\,\,4) = (2\,\,i_2) \neq (1\,\,i_2).$ Consequently, $ xyx\neq y.$. Therefore, $ zy $ and $ zxyx$ are distinct elements in $ zH$. It is clear that $ cl(zy) = cl(zxyx)$ as $ xz=zx$. Hence, we are done.
\vspace{0.1cm}

\noindent\textbf{Subcase C:} When $ |\Delta(z)\cap\Delta(y)| = 2.$

\noindent Suppose $ x $ and $ y $ do not share a transposition. Then $ y $ contains one of transpositions $ (1\,\,3),~ (2\,\,4),~(1\,\,4),~(2\,\,3)$. Consequently, $ xyx \neq  y $. Therefore, $ zy $ and $ zxyx$ are distinct elements in $ zH$. It is clear that $ cl(zy) = cl(zxyx)$ as $ xz=zx$. Hence, we are done.

\vspace{0.1cm}

\noindent Suppose $ x $ and $ y $ share a transposition. Without loss of generality, let $ y = (1\,\,2)(i_3\,\,i_4)\cdots(i_{2r-1}\,\,i_{2r}) $. Then $ zy = (1\,\,4\,\,2\,\,3)(i_3\,\,i_4)\cdots(i_{2r-1}\,\,i_{2r})$ and $ zxy = (1\,\,3\,\,2\,\,4)(i_3\,\,i_4)\cdots(i_{2r-1}\,\,i_{2r})$. This shows that $ cl(zy) = cl(zxy) $. Since $ zy,~zxy $ are distinct elements in $ zH $, therefore, we are done.

\vspace{0.1cm}

\noindent\textbf{Subcase D:} When $  |\Delta(z)\cap \Delta(y)| = 3 .$

\noindent Then there exists a transposition $(a_1\,\,a_2) $ of $ y $ such that $ a_1 \in\Delta(z) $ and $ a_2\notin \Delta(z)$. Consequently, $ x(a_2) = a_2$ as $ \Delta(z)= \Delta(x)$. Suppose $ xyx = y$. Then $ (xyx)(a_2) = y(a_2) = a_1$. It implies that $ x(a_1) = a_1 $, which is a contradiction as $ a_1\in \Delta(x).$ Hence, $ xyx \neq y $. Thus, $ zy \neq z(xyx)$. Let $zy\in cl(g) $ for some $ g\in S_n $. Then $ x(zy)x^{-1} \in cl(g)$. Since $ xz = zx $, therefore, $ z(xyx) \in cl(g)$. Consequently, $ cl(zy) = cl(zxyx).$ Therefore, we are done.

\vspace{0.1cm}

\noindent\textbf{Subcase E:} When $  |\Delta(z)\cap \Delta(y)| = 4 .$

\noindent Suppose $ y $ contains one of $(1\,\,4)(2\,\,3)$, $  (1\,\,2)(3\,\,4)$. Then $ cl(zy) = cl(y)$. Therefore, we are done.  

\noindent Suppose $ z $ acts nontrivially atleast three transpositions of $ y $. Then there exists a transposition $(a_1\,\,a_2) $ of $ y $ such that $ a_1 \in\Delta(z) $ and $ a_2\notin \Delta(z)$. It follows from the aforementioned subcase that $ xyx\neq y $. Then $ zy $ and $ zxyx $ are distinct elements in $ zH$ with $ cl(zy) = cl(zxyx)$ as $ xz = zx$. Hence, we are done. 
\vspace{0.05cm}

\noindent Suppose $ z $ acts two transpositions of $ y $. Without loss of generality, let $  y = (1\,\,3)(2\,\,4)(i_1\,\,i_2)\cdots(i_{2r'-1}\,\,i_{2r'}).$ Then $ zy = (i_1\,\,i_2)\cdots(i_{2r'-1}\,\,i_{2r'})$. Here, \begin{equation}\label{eqn01}
2+r' = r.
\end{equation}
 If $ r = 2 $, then it follows from equation (\ref{eqn01}) that $ r' = 0 .$ Consequently, $ y = x$. Thus, we arrive at a contradiction. Hence, $ r >2$.
 
\noindent Suppose $ r $ is even. Then it follows from equation (\ref{eqn01}) that $r'$ is even. Therefore, $ zy $ is a square element. Consequently, we are done.

\vspace{0.1cm}

\noindent\textbf{Case II:} Suppose $ (i\,\,k)(j\,\,l),~(i\,\,l)(j\,\,k) \in H $ for every $ (i\,\,j)(k\,\,l)\in H $.

\noindent Without loss of generality, let $ x = (1\,\,2)(3\,\,4)\in H $ and $z = (1\,\,2) \in S_n $. Then $xz =zx $ and $ cl(z) = cl(zx) $. It is clear that $ z\notin H $ as $ I_{min}(H) = 2 $. Let $ y \in I_r(H).$ 

\vspace{0.1cm}

\noindent To proceed further, we consider the following subcases.

\vspace{0.1cm}

\noindent\textbf{Subcase A:} When $ |\Delta(z)\cap\Delta(y)| = 0 .$

\vspace{0.1cm}

\noindent Suppose $ |\Delta(x)\cap\Delta(y)| = 0 $, then $ cl(zy ) = cl(zxy)$. Consequently, we are done.
\vspace{0.1cm}

\noindent Suppose $ |\Delta(x)\cap\Delta(y)| = 1 $. Then it follows from Remark \ref{rmk01} that $ xy \neq yx $. Therefore, $zy$ and $ zxyx $ are distinct elements in $ zH $ and $ cl(zy) = cl(zxyx) $ as $ xz = zx $. Consequently, we are done. 
\vspace{0.1cm}

\noindent Suppose $ |\Delta(x)\cap\Delta(y)| = 2  $. If $ x$ acts nontrivially on two transpositions of $ y $, then $ xy\neq yx $. Therefore, $ xy \neq yx $ due to Remark \ref{rmk01}. Then $zy$ and $ zxyx $ are distinct elements in $ zH $ and $ cl(zy) = cl(zxyx) $ as $ xz = zx $. Consequently, we are done. 

\noindent Suppose $ x$ and $ y $ share exactly one transposition. Without loss of generality, let $ y = (3\,\,4)(i_1\,\,i_2)\cdots(i_{2r'-1}\,\,i_{2r'}) $. Then \begin{equation}\label{eqn02}
    1+r' = r.
\end{equation}  
Now, $ zy = (1\,\,2)(3\,\,4)(i_1\,\,i_2)\cdots(i_{2r'-1}\,\,i_{2r'})$. If $ r $ is odd, then $ zy $ is square. Consequently, we are done. Suppose $ r $ is even. Then it follows from equation (\ref{eqn02}) that $ r' $ is odd. Since $xy = (1\,\,2)(i_1\,\,i_2)\cdots(i_{2r'-1}\,\,i_{2r'}) \in H $ and $ (1\,\,3)(2\,\,4)\in H$, therefore, $ (1\,\,4\,\,2\,\,3)(i_1\,\,i_2)\cdots(i_{2r'-1}\,\,i_{2r'})\in H$. Consequently, $ z[(1\,\,4\,\,2\,\,3)(i_1\,\,i_2)\cdots(i_{2r'-1}\,\,i_{2r'})]\in zH $. This implies that $ (1\,\,4)(2\,\,3)(i_1\,\,i_2)\cdots(i_{2r'-1}\,\,i_{2r'})\in zH$. It shows that $ cl(zy) = cl((1\,\,4)(2\,\,3)(i_1\,\,i_2)\cdots(i_{2r'-1}\,\,i_{2r'}))$. Therefore, we are done. 

\vspace{0.1cm}

\noindent\textbf{Subcase B:} When $ |\Delta(z)\cap\Delta(y)| = 1 $.

\noindent Suppose $ |\Delta(x)\cap\Delta(y)| = 1 $. Then it follows from Remark \ref{rmk01} that $ xy \neq yx $. Therefore, $zy$ and $ zxyx $ are distinct elements in $ zH $ and $ cl(zy) = cl(zxyx) $ as $ xz = zx $. Consequently, we are done. 
\vspace{0.1cm}

\noindent Suppose $ |\Delta(x)\cap\Delta(y)| = 2  $. If $ x$ acts nontrivially on two transpositions of $ y $, then $ xy\neq yx $. Then it follows from Remark \ref{rmk01} that $ xy \neq yx $. Therefore, $zy$ and $ zxyx $ are distinct elements in $ zH $ and $ cl(zy) = cl(zxyx) $ as $ xz = zx $. Consequently, we are done.  

\noindent Suppose $ x$ and $ y $ share exactly one transposition. Without loss of generality, let $ y = (1\,\,2)(i_1\,\,i_2)\cdots(i_{2r'-1}\,\,i_{2r'}) $, where \begin{equation*}
      1+r' = r.
\end{equation*}  
Then $ zy = (i_1\,\,i_2)\cdots(i_{2r'-1}\,\,i_{2r'})$. Suppose $ r $ is odd, then $ zy $ is square. Consequently, we are done. Suppose $ r $ is even. Let us consider the involution $ w = (3\,\,4)(i_1\,\,i_2) $. Then $ w(zy) = (3\,\,4)(i_3\,\,i_4)\cdots(i_{2r'-1}\,\,i_{2r'})$ and $ w(zxy) = (1\,\,2)(i_3\,\,i_4)\cdots(i_{2r'-1}\,\,i_{2r'}) $. This shows that $ cl(wzy) = cl(wzxy)$. Therefore, we are done.
\vspace{0.1cm}

\noindent Suppose $ |\Delta(x)\cap\Delta(y)| = 3  $. Therefore, $ xy \neq yx $ due to Remark \ref{rmk01}. Then $zy$ and $ zxyx $ are distinct elements in $ zH $ and $ cl(zy) = cl(zxyx) $ as $ xz = zx $. Consequently, we are done. 

\vspace{0.1cm}

\noindent\textbf{Subcase C:} When $ |\Delta(z)\cap\Delta(y)| = 2 $.

\noindent \noindent Suppose $ |\Delta(x)\cap\Delta(y)| = 2  $. Suppose $ z $ acts nontrivially on two transpositions of $ y $. Without loss of generality, let $ y = (1\,\,i_2)(2\,\,i_4)(i_5\,\,i_6)\cdots(i_{2r-1}\,\,i_{2r}).$ Then $ xy \neq yx $, which is a contradiction as $ xy=yx.$ Hence, $ z $ and $ y $ share a common transposition. Without loss of generality, let $ y = (1\,\,2)(i_3\,\,i_4)\cdots(i_{2r-1}\,\,i_{2r}).$ Then $ cl(zy) = cl(xy) $. Consequently, we are done.
\vspace{0.1cm}

\noindent Suppose $ |\Delta(x)\cap\Delta(y)| = 3  $. Therefore, $ xy \neq yx $ due to Remark \ref{rmk01}. Then $zy$ and $ zxyx $ are distinct elements in $ zH $ and $ cl(zy) = cl(zxyx) $ as $ xz = zx $. Consequently, we are done. 
\vspace{0.1cm}

\noindent Suppose $ |\Delta(x)\cap\Delta(y)| = 4  $. 

\noindent Suppose $ x $ acts nontrivially on three transpositions of $ y $. Then $ xyx\neq y $. Therefore, $zy$ and $ zxyx $ are distinct elements in $ zH $ and $ cl(zy) = cl(zxyx) $ as $ xz = zx $. Consequently, we are done.

\noindent Suppose $ x $ acts nontrivially on two transpositions of $ y $. Without loss of generality, let $ y = (1\,\,2)(3\,\,4)(i_1\,\,i_2)\cdots(i_{2r'-1}\,\,i_{2r'})$, where \begin{equation*}
    2+r' = r .
\end{equation*}

\noindent Then $ zy = (3\,\,4)(i_1\,\,i_2)\cdots(i_{2r'-1}\,\,i_{2r'})$ and $ zxy = (1\,\,2)(i_1\,\,i_2)\cdots(i_{2r'-1}\,\,i_{2r'})$. This shows that $ cl(zy) =  cl(zxy) $. Consequently, we are done and this completes this case.

\vspace{0.1cm}

\noindent It follows from the above cases that our claim is justified. Hence, every element of $zH$ is square in $S_n$ for some $ z\in S_n\setminus H $; or for every non-square element $\sigma \in zH$, there exists $w\in S_n $ such that $|Cl(\sigma)\cap wH| > 1$ and/or $ cl(\sigma)\cap H \neq\emptyset$. Consequently, $H$ is not a perfect subgroup code of $S_n$ in the Cayley sum graph due to Lemma \ref{lmn3.2} and this completes the proof.
\end{proof}

\begin{theorem}
Let $ H $ be a nontrivial even order subgroup of $A_n $ such that $ I_{min}(H) = 2$. Then $ H $ is not a perfect subgroup code of $ A_n $ in the Cayley sum graph.    
\end{theorem}
\begin{proof}
let $ J(H) = \{ 1_{S_n}, y_1, y_1^{-1}, \dots, y_m, y_m^{-1}\}.$ It follows from Lemma \ref{lm 1} that $| cl(zy_i)\cap zH|>1$ for any $ z \in I(A_n)$, for $ 1\leq i \leq m $.

\vspace{0.1cm}

\noindent We claim that for every non square element $\sigma \in zH $ for some $ z\in I(A_n\setminus H)$, there exists $ w \in S_n $ such that $ |cl(\sigma)\cap wH|>1 $ or $ cl(\sigma)\cap H \neq \emptyset$ or both. To proceed further, we consider the following cases.

\vspace{0.1cm}

\noindent\textbf{Case I:} Suppose $ (i\,\,k)(j\,\,l),~(i\,\,l)(j\,\,k) \in H $ for every $ (i\,\,j)(k\,\,l)\in H $.
\vspace{0.1cm}

\noindent If $ H $ contains all product two  transpositions, then $ H = A_n$ due to Lemma \ref{lmn2.6}, which is contradiction. Therefore, there exists $(a\,\,b)(c\,\,d) \in A_n $ such that $ (a\,\,b)(c\,\,d) \notin H.$ Without loss of generality, let $z = (a\,\,b)(c\,\,d) = (1\,\,2)(3\,\,4).$ Let $ y = (i_1\,\,i_2)\cdots(i_{2r-1}\,\,i_{2r})\in I_r(H)$. Since $ H $ is in $ A_n $, therefore, $ r $ is always even. To proceed further, we consider the following subcases.
\vspace{0.1cm}

\noindent\textbf{Subcase A:} When $ |\Delta(z)\cap\Delta(y)| = 0$.

\noindent Then $ zy $ is a square element in $ A_n $ as $ r$ is even. Consequently, we are done.
\vspace{0.1cm}

\noindent\textbf{Subcase B:} When $ |\Delta(z)\cap\Delta(y)| = 1 $.

\noindent Without loss of generality, let $ y = (1\,\,i_2)(i_3\,\,i_4)\cdots (i_{2r-1}\,\,i_{2r}).$ Then $ zy = (1\,\,2)(3\,\,4)(1\,\,i_2)\cdots(i_{2r-1}\,\,i_{2r}) = (1\,\,i_2\,\,2)(3\,\,4)(i_3\,\,i_4)\cdots ( i_{2r-1}\,\,i_{2r} )$. This shows that $ zy $ is a square element in $ A_n $ as $ r $ is even and any $ 3 $ cycle is square element. Hence, we are done. 
\vspace{0.1cm}

\noindent\textbf{Subcase C:} When $ |\Delta(z)\cap\Delta(y)| = 2 $.

\noindent Suppose $ z $ and $ y $ share a transposition. Then $ zy $ is a square element in $ A_n.$ Consequently, we are done.
\vspace{0.1cm}

\noindent Suppose $ (1\,\,2),~(3\,\,4) $ acts nontrivially on distinct transpositions of $ y$. Without loss of generality, let $ y = (1\,\,i_2)(3\,\,i_4)(i_5\,\,i_6)\cdots(i_{2r-1}\,\,i_{2r})$. Then $ zy = (1\,\,i_2\,\,2)(3\,\,i_4\,\,4)(i_5\,\,i_6)\cdots(i_{2r-1}\,\,i_{2r}).$ Consequently, $ zy $ is square element in $ A_n $ as $ r $ is even. Therefore, we are done.
\vspace{0.1cm}

\noindent Suppose one of $ (1\,\,2),~(3\,\,4)$ acts nontrivially on two transpositions of $ y $. Without loss of generality, let $ y = (1\,\,i_2)(2\,\,i_4)(i_5\,\,i_6)\cdots(i_{2r-1}\,\,i_{2r}).$ Then $ zy = (1\,\,i_2\,\,2\,\,i_4)(3\,\,4)(i_5\,\,i_6)\cdots(i_{2r-1}\,\,i_{2r})$. Let us consider the element $ w = (i_5\,\,i_7\,\,i_6\,\,i_8)[(1\,\,2)(3\,\,4)(i_2\,\,i_4)][(i_9\,\,i_{11})(i_{10}\,\,i_{12})\cdots(i_{2r-2}\,\,i_{2r})].$ The $ wy = (i_5\,\,i_8\,\,i_6\,\,i_7)[(1\,\,i_4)(2\,\,i_2)(3\,\,4)][(i_9\,\,i_{12})(i_{10}\,\,i_{11})\cdots(i_{2r-2}\,\,i_{2r-1})(i_{2r-3}\,\,i_{2r})].$ This shows that $ cl(w) = cl(wy) = cl(zy).$ Thus, we are done.
\vspace{0.1cm}

\noindent\textbf{Subcase D:} When $ |\Delta(z)\cap\Delta(y)| = 3 $.

\noindent Suppose $ z $ and $ y $ share a transposition. Without loss of generality, let $ y = (1\,\,2)(3\,\,i_4)(i_5\,\,i_6)\cdots(i_{2r-1}\,\,i_{2r}).$ Then $ zy = (3\,\,i_4\,\,4)(i_5\,\,i_6)\cdots(i_{2r-1}\,\,i_{2r}).$ Consequently, $ zy $ is a square in $ A_n.$ Therefore, we are done.  
\vspace{0.1cm}

\noindent Suppose $ z $ and $ y $ do not share a transposition such that $ z $ acts nontrivially on two transpositions of $ y$. Without loss of generality, let $ y = (1\,\,i_2)(2\,\,3)(i_5\,\,i_6)\cdots(i_{2r-1}\,\,i_{2r}) $. Then $ zy = (1\,\,i_2\,\,2\,\,4\,\,3)(i_5\,\,i_6)\cdots(i_{2r-1}\,\,i_{2r}).$ Consequently, $ zy $ is a square element in $ A_n $ as $ r $ is even and odd-order element is square in $A_n$. 

\vspace{0.1cm}

\noindent Suppose $ z $ and $ y $ do not share a transposition such that $ z $ acts nontrivially on three transpositions of $ y$. Without loss of generality, let $ y = (1\,\,i_2)(2\,\,i_4)(3\,\,i_6)(i_7\,\,i_8)\cdots(i_{2r-1}\,\,i_{2r}).$ Then $ zy = (1\,\,i_2\,\,2\,\,i_4)(3\,\,i_6\,\,4)(i_7\,\,i_8)\cdots (i_{2r-1}\,\,i_{2r}).$ Let us consider the element $ w = (1\,\,3)(2\,\,i_8\,\,i_6)(4\,\,i_7\,\,i_6\,\,i_2)[(i_9\,\,i_{11})\cdots(i_{2r-2}\,\,i_{2r})].$ Then $ wy = (3\,\,i_2)(2\,\,i_4\,\,i_8)(1\,\,4\,\,i_7\,\,i_6)[(i_9\,\,i_{12})(i_{10}\,\,i_{11})\cdots (i_{2r-2}\,\,i_{2r-1})(i_{2r-3}\,\,i_{2r})].$ This shows that $ cl(wy)=cl(w)=cl(zy).$ Consequently, we are done. 
\vspace{0.1cm}

\noindent\textbf{Subcase E:} When $ |\Delta(z)\cap\Delta(y)| = 4 $.

\noindent Suppose $z$ acts nontrivially on two transpositions of $ y $, then $ zy $ is an involution in $ A_n $. Consequently, $zy$ is a square element in $ A_n$. Therefore, we are done.
\vspace{0.1cm}

\noindent  Suppose $z$ acts nontrivially on three transpositions of $ y $ such that $ z $ and $ y $ share exactly one transpositions. Without loss of generality, let $ y = (1\,\,2)(3\,\,i_4)(4\,\,i_6)(i_7\,\,i_8)\cdots(i_{2r-1}\,\,i_{2r}).$ Then $ zy = (3\,\,i_4\,\,4\,\,i_6)(i_7\,\,i_8)\cdots(i_{2r-1}\,\,i_{2r}).$ Let us consider the element $ w = (1\,\,2)(3\,\,4\,\,i_4\,\,i_6)[(i_9\,\,i_{11})\cdots(i_{2r-2}\,\,i_{2r})].$ Then $ wy = (3\,\,i_6\,\,i_4\,\,4)(i_7\,\,i_8)[(i_9\,\,i_{12})(i_{10}\,\,i_{11})\cdots(i_{2r-2}\,\,i_{2r-1})(i_{2r-3}\,\,i_{2r})].$ This shows that $ cl( wy) = cl(w)=cl(zy).$ Consequently, we are done. 
\vspace{0.1cm}

\vspace{0.1cm}

\noindent Suppose $z$ acts nontrivially on four transpositions of $ y $. Without loss of generality, let $ y =(1\,\,i_2)(2\,\,i_4)(3\,\,i_6)(4\,\,i_8)(i_9\,\,i_{10})\cdots(i_{2r-1}\,\,i_{2r}).$ Then $ zy = (1\,\,i_2\,\,2\,\,i_4)(3\,\,i_6\,\,4\,\,i_8)(i_9\,\,i_{10})\cdots(i_{2r-1}\,\,i_{2r})$. Since $ r $ is even, therefore, $ zy $ is square element in $ A_n.$ Consequently, we are done.

\vspace{0.2cm}
\noindent\textbf{Case II:} Suppose there exists $ (i\,\,j)(k\,\,l)\in H $ such that $ (i\,\,k)(j\,\,l)\notin H $. 

\vspace{0.1cm}

\noindent This case reduces to a part of Case I in the proof of Theorem \ref{thm004}.

\vspace{0.1cm}

\noindent\noindent It follows from the above cases that our claim is justified. Consequently, $H$ is not a perfect subgroup code of $A_n$ in the Cayley sum graph due to Lemma \ref{lmn3.2} and this completes the proof.
\end{proof}  

\vspace{0.1 cm}

\begin{theorem}\label{thm5}
Let $ H $ be a nontrivial even order subgroup of $ S_n $ such that $ I_{min}(H)>2$. Then $ H $ is not a perfect subgroup code of $ S_n $ in the Cayley sum graph.
\end{theorem}
\begin{proof}
let $  I_{min}(H) = k $ and $ J(H) = \{ 1_{S_n}, y_1, y_1^{-1}, \dots, y_m, y_m^{-1}\}.$ It follows from Lemma \ref{lm 1} that $| cl(zy_i)\cap zH|>1$ for any $ z \in I(S_n)$, for $ 1\leq i \leq m $.

\vspace{0.1cm}

 \noindent We claim that for every non square element $\sigma \in zH $ for some $ z\in I(S_n\setminus H)$, there exists $ w \in S_n $ such that $ |cl(\sigma)\cap wH|>1 $ or $ cl(\sigma)\cap H \neq \emptyset $ or both.

\vspace{0.1cm} 
\noindent Suppose $ x \in I_{k}(H)$ and $ y \in I_r(H)$. Without loss of generality, let $ x = (1\,\,2)(3\,\,4) \cdots (2k-1\,\,2k) $.  Let $ z = (1\,\,3)(2\,\,4)$. Then $ z\notin H$ as $ k>2$. It is clear that $ xz = zx$. Consequently, $ cl(zy) = cl(zxyx) $.
\vspace{0.3cm}

\noindent Without loss of generality, let $ x = (1\,\,2)(3\,\,4)x_1x_2x_3 $,  where $ x_1 =  (j_1\,\,j_2)(j_3\,\,j_4)\cdots (j_{2l-1}\,\,j_{2l}) $,  $x_2 =  (t_1\,\,t_2)(t_3\,\,t_4)\cdots (t_{2l'-1}\,\,j_{2l'})$, and $ x_3 = (s_1 \,\,s_2)\cdots(s_{2l''-1}\,\,s_{2l''})$ such that transpositions of $ x_1$ are common transpositions of $ y $ and transpositions of $ x_3$ are not in $ y $ but $ \Delta(x_3)\subset \Delta(y)$ with $ x_3y = yx_3$. 
\noindent Here, \begin{equation}\label{eqn1}
        2+l+l' +  l''  = k. 
\end{equation} It follows from Lemma \ref{lmn2.10} that $ l''$ is always even. Let $ y\in I(H)$ such that $ \mathcal{N}(y) = r$. To proceed further, we consider the following cases.

 \vspace{0.17cm}
 
\noindent\textbf{Case I:} When $ |\Delta(z)\cap \Delta(y)| = 0.$

\noindent Without loss of generality, let $ y  = y_1y_2y_3$, where $ y_1 = [(j_1\,\,j_2)\cdots (j_{2l-1}\,\,j_{2l})]$, $y_2 = [(s_1 \,\,s_3)\cdots(s_{2l''-2}\,\,s_{2l''})]$, and $y_3 = [(i_1\,\,i_2)\cdots (i_{2r'-1}\,\,i_{2r'})]$ such that $ \{i_1, i_2,\dots,i_{2r'}\}\notin \Delta(x)$. Since $ y \in I_r(H),$ therefore, 
 \begin{equation}\label{eqn2}
        l+l''+r'=r. 
\end{equation}

\noindent It is clear that $zy\in I_{r+2}(S_n\setminus H).$ Consequently, $ \mathcal{N}(zy) = r+2.$
\vspace{0.1cm}

\noindent Suppose $ r$ is even, then $ zy$ is a square element in $S_n$. Consequently, we are done. 
\vspace{0.2cm}

\noindent Suppose $ r $ is odd. It is clear that $ l, l', l''\leq k $. We divide this case into four parts. 
\vspace{0.05cm}

\noindent \textbf{Subcase A:} When $ k $ is even and $ l$ is odd.

\noindent It follows from equation (\ref{eqn1}) that $ l' $ is odd. Let us consider the involution $ w = (1\,\,2)(3\,\,4)w_1w_2w_3w_4$, where $ w_1 = (j_1\,\,j_3)(j_2\,\,j_4)\cdots (j_{2l-4}\,\,j_{2l-2})$, $ w_2 = [(t_1\,\,t_3)(t_2\,\,t_4)\cdots (t_{2l'-4}\,\,t_{2l'-2})](t_{2l'-1}\,\,t_{2l'})$, $w_3 = (i_1\,\,i_2)(i_3\,\,i_4)\cdots (i_{2(r'-l'+1)-1}\,\,i_{2(r'-l'+1)})$, and $ w_4 = (s_1\,\,s_3)\cdots(s_{2l''-2}\,\,s_{2l''})$. If $ w \in H $, then we are done. Suppose $ w\notin H$. Now, $ wz = (1\,\,4)(2\,\,3)w_1w_2w_3w_4$ and $ \mathcal{N}(w) = 2+(l-1)+l'+ r'-l'+1 + l'' = r + 2 $ by equation (\ref{eqn2}). It is clear that $ cl(wz)=cl(w)$ as $ \mathcal{N}(wz) = \mathcal{N}(w)$. Now, $w(zx) = (1\,\,3)(2\,\,4)u_1u_2 u_3u_4$, where $ u_1 = [(j_1\,\,j_4)(j_2\,\,j_3)\cdots(j_{2l-5}\,\,j_{2l-2})(j_{2l-4}\,\,j_{2l-3})](j_{2l-1}\,\, j_{2l})$, $  u_2 = (t_1\,\,t_4)(t_2\,\,t_3)\cdots (t_{2l'-4}\,\,t_{2l'-3})(t_{2l'-5}\,\,t_{2l'-2})$, $  u_3 = w_3$, and $ u_4 = (s_1\,\,s_4)(s_2\,\,s_3)\cdots(s_{2l''-2}\,\,s_{2l''-1}) (s_{2l''-3}\,\,s_{2l''})$. Then $ \mathcal{N}(wzx) = 2+l +l'-1+r'-l'+1+l'' = r+2 = \mathcal{N}(zy).$ This shows that $ cl(wzx) = cl(wz) = cl(w) = cl(zy)$ and consequently, we are done.

\vspace{0.1cm}
\noindent \textbf{Subcase B:}  When $ k $  is even and $ l $ is even. 

\noindent Then it follows from equations (\ref{eqn1}) and (\ref{eqn2}) that $ l' $ is even and $ r' $ is odd as $ r$ is odd. Let us consider the involution $ w = (1\,\,2)(3\,\,4)w_1w_2w_3w_4$, where $ w_1 = (j_1\,\,j_3)(j_2\,\,j_4)\cdots (j_{2l-3}\,\,j_{2l-1})(j_{2l-2}\,\,j_{2l})$, $ w_2 = (t_1\,\,t_3)(t_2\,\,t_4)\cdots (t_{2l'-2}\,\,t_{2l'})$, $ w_3 = [(i_1\,\,i_2)\cdots (i_{2(r'-l')-1}\,\,i_{2(r'-l')})]$, and $ w_4 = (s_1\,\,s_3)\cdots(s_{2l''-2}\,\,s_{2l''})$. Then $  wz = (1\,\,4)(2\,\,3)w_1w_2w_3w_4$ and $ wzx = (1\,\,3)(2\,\,4)u_1u_2u_3u_4$, where $ u_1 = (j_1\,\,j_4)(j_2\,\,j_3)\cdots(j_{2l-2}\,\,j_{2l-1})(j_{2l-3}\,\,j_{2l})$, $ u_2 = (t_1\,\,t_4)(t_2\,\,t_3)\cdots(t_{2l'-2}\,\,t_{2l'-1})(t_{2l'-3}\,\,t_{2l'})$, $ u_3 = w_3 $, and $ u_4 = (s_1\,\,s_4)(s_2\,\,s_3)\cdots(s_{2l''-2}\,\,s_{2l''-1})(s_{2l''-3}\,\,s_{2l''})$. Then $ \mathcal{N}(wz) = 2 + l + l' + r'-l' + l'' =\mathcal{N}(zy).$ Similarly, we get $ \mathcal{N}(wzx) = \mathcal{N}(wz)$. Hence, $ cl(wz) = cl(wzx) = cl(zy)$. Consequently, we are done.

\vspace{0.1cm}
\noindent \textbf{Subcase C:} When $ k $ is odd and $ l $ is odd. 

\noindent Then it follows from equations (\ref{eqn1}) and (\ref{eqn2}) that $ l'$ is even and $ r'$ is odd as $ r$ is odd. Now, $ zxy = (1\,\,4)(2\,\,3)(t_1\,\,t_2)\cdots (t_{2l'-1}\,\,t_{2l'})(i_1\,\,i_2)\cdots (i_{2r'-1}\,\,i_{2r'})(s_1\,\,s_4)(s_2\,\,s_3)\cdots(s_{2l''-3}\,\,s_{2l''})$. 

\noindent Let us consider the involution $ w = (1\,\,2)(3\,\,4)w_1w_2w_3w_4 $, where $ w_1 = (j_1\,\,j_3)(j_2\,\,j_4)\cdots (j_{2l-5}\,\,j_{2l-3})(j_{2l-4}\,\,j_{2l-2})$, $ w_2 = (i_1\,\,i_2)(i_3\,\,i_4)\cdots(i_{2l'-1}\,\,i_{2l'}) $, $ w_3 = [(i_{2l'+1}\,\,i_{2l'+3})(i_{2l'+2}\,\,i_{2l'+4})\cdots(i_{2r'-4}\,\,i_{2r'-2})](i_{2r'-1}\,\,j_{2l-1})(i_{2r'}\,\,j_{2l}) $ and $ w_4 = x_3$. Then $ \mathcal{N}(w) = 2+l-1+r'+1 +l''= r+2 $ by equation (\ref{eqn2}). It is clear that $ \mathcal{N}(wz) = \mathcal{N}(w) = r+2.$ If $ w \in H $, then we are done. Suppose $ w\notin H $. Now, $ w(zxy) = (1\,\,3)(2\,\,4)u_1u_2u_3u_4$, where $ u_1 = (j_1\,\,j_3)(j_2\,\,j_4)\cdots(j_{2l-5}\,\,j_{2l-3})(j_{2l-4}\,\,j_{2l-2})$, $ u_2 = [(i_{2l'+1}\,\,i_{2l'+4})(i_{2l'+2}\,\,i_{2l'+3})\cdots (i_{2r'-4}\,\,i_{2r'-3}) (i_{2r'-5}\,\,i_{2r'-2})](i_{2r'}\,\,j_{2l-1})(i_{2r'-1}\,\,j_{2l})$, $ u_3 = (t_1\,\,t_2)(t_3\,\,t_4)\cdots(t_{2l'-1}\,\,t_{2l'})$, and $ u_4 = (s_1\,\,s_3)(s_2\,\,s_4)\cdots(s_{2l''-2}\,\,s_{2l''})$. Then $ \mathcal{N}(wzxy) = 2+l-1+(r'-l'+1)+l'+l'' =2+l+r'+l'' = 2 + r $ by equation (\ref{eqn2}). This shows that $ cl(zy) = cl(wz) = cl(wzxy).$ Therefore, we are done. 

\vspace{0.1cm}
\noindent \textbf{Subcase D:} When $ k $ is odd and $ l $ is even. 

\noindent Then it follows from equations (\ref{eqn1}) and (\ref{eqn2}) that $ l'$ is odd and $ r'$ is odd. Let us consider the involution $ w = (1\,\,2)(3\,\,4)w_1w_2w_3w_4$, where $ w_1 = (j_1\,\,j_3)(j_2\,\,j_4)\cdots (j_{2l-3}\,\,j_{2l-1})(j_{2l-2}\,\,j_{2l})$, $ w_2 = (i_1\,\,i_2)(i_3\,\,i_4)\cdots(i_{2l'-1}\,\,i_{2l'}) $, $ w_3 = (i_{2l'+1}\,\,i_{2l+3})(i_{2l+2}\,\,i_{2l+4})\cdots(i_{2r'-3}\,\,i_{2r'-1})(i_{2r'-2}\,\,i_{2r'}) $, and $ w_4 = (s_1\,\,s_3)\cdots(s_{2l''-2}\,\,s_{2l''})$. Then $ wz = (1\,\,4)(2\,\,3)w_1w_2w_3w_4$ and $ \mathcal{N}(w) = 2+l+r'+l'' = r+2 $ by equation (\ref{eqn2}). If $ w\in H $, then we are done. Suppose $ w\notin H$. It is easy to see that $ \mathcal{N}(wz) = \mathcal{N}(w) = r+2 = \mathcal{N}(zy)$. Now, $ w(zxy) = (1\,\,3)(2\,\,4)u_1u_2u_3u_4$, where $ u_1 = (j_1\,\,j_3)(j_2\,\,j_4)\cdots (j_{2l-2}\,\,j_{2l})$, $ u_2 = (i_{2l'+1}\,\,i_{2l'+4})(i_{2l'+2}\,\,i_{2l'+3})\cdots(i_{2r'-2}\,\,i_{2r'-1})(i_{2r'-3}\,\,i_{2r'})$, $ u_3 = (t_1\,\,t_2)(t_3\,\,t_4)\cdots(t_{2l'-1}\,\,t_{2l'})$, and $ u_4 = (s_1\,\,s_2)\cdots(s_{2l''-1}\,\,s_{2l''})$ Then $ \mathcal{N}(wzxy) = 2+l+l'+l''+r'-l' = 2+r $ by equation (\ref{eqn2}). This shows that $ cl(zy)= cl(wzxy)=cl(wz)$ and consequently, we are done. 

\vspace{0.1cm}
\noindent\textbf{Case II:} When $  |\Delta(z)\cap \Delta(y)| = 1 .$

\noindent Without loss of generality, let $ i_1 = 1 $. Now, $ xyx^{-1} = (x(i_1)\,\,x(i_2))(x(i_3) \,\, x(i_4)) \cdots (x(i_{2r-1})\,\, x(r_{2r})).$ Suppose $ xyx^{-1} = y $. Now, $ (xyx^{-1})(1) = y (1) = i_2 $. This implies that $xy(2) = i_2 $. Consequently, $ x(2) = i_2 $. Then $ 1 = i_1 = i_2 $, which is contradiction as $ xy = yx.$ Hence, $ xyx^{-1} \neq x $. Thus, $ zy \neq z(xyx)$. Let $zy\in cl(g) $ for some $ g\in S_n $. Then $ x(zy)x^{-1} \in cl(g)$. Since $ xz = zx $, therefore, $ z(xyx) \in cl(g)$. Consequently, $ cl(zy) = cl(zxyx).$ Therefore, we are done.

\vspace{0.1cm}
\noindent\textbf{Case III:} When $  |\Delta(z)\cap \Delta(y)| = 2 .$

\noindent Suppose two transpositions of $ z $ act non trivially on two transpositions of $ y $. Without loss of generality, let $ y =(1\,\,i_2)(2\,\,i_4)(i_5\,\,i_6)\cdots(i_{2r-1}\,\,i_{2r}).$ Since $ (1\,\,2)(3\,\,4)(1\,\,i_2)(2\,\,i_4) \neq (1\,\,i_2)(2\,\,i_4)(1\,\,2)(3\,\,4)$, therefore, $ xyx \neq y.$ Let $zy\in cl(g) $ for some $ g\in S_n $. Then $ x(zy)x^{-1} \in cl(g)$. Since $ xz = zx $, therefore, $ z(xyx) \in cl(g)$. Consequently, $ cl(zy) = cl(zxyx).$ Therefore, we are done. 

\vspace{0.1cm}

\noindent Suppose one transposition of $ z $ acts non trivially on two transpositions of $ y $. Without loss of generality, let $ y =(1\,\,i_2)(3\,\,i_4)(i_5\,\,i_6)\cdots(i_{2r-1}\,\,i_{2r}).$ Since $ (1\,\,2)(3\,\,4)(1\,\,i_2)(3\,\,i_4) \neq (1\,\,i_2)(3\,\,i_4)(1\,\,2)(3\,\,4)$, therefore, $ xyx \neq y.$ Let $zy\in cl(g) $ for some $ g\in S_n $. Then $ x(zy)x^{-1} \in cl(g)$. Since $ xz = zx $, therefore, $ z(xyx) \in cl(g)$. Consequently, $ cl(zy) = cl(zxyx).$ Therefore, we are done.

\vspace{0.1cm}

\noindent Suppose $ z $ and $ y $ share exactly one common transposition. Then $ cl(zy) = cl(y)$ as $ \mathcal{N}(z) = 2 $. Consequently, we are done. 
\vspace{0.1cm}

\begin{equation}\label{eqn3}
        1+l+l''+r'=r. 
\end{equation}

\vspace{0.1cm}
\noindent Now, $ zy = (1\,\,4\,\,2\,\,3)y_1y_2y_3$. Then $ \mathcal{N}(zy)= r-1.$ To proceed further, we consider the following subcases.
\vspace{0.1cm}

\noindent \textbf{Subcase A:} When $ k $ is even.

\noindent Suppose $ l $ is odd. Then it follows from equation (\ref{eqn1}) that $ l' $ is odd. Let us consider the involution $ w = (1\,\,2)w_1w_2w_3 w_4w_5$, where $ w_1 = (j_1\,\,j_3)(j_2\,\,j_4)\cdots (j_{2l-5}\,\,j_{2l-3})(j_{2l-4}\,\,j_{2l-2})$, $ w_2 = [(t_1\,\,t_3)(t_2\,\,t_4)\cdots (t_{2l'-4}\,\,t_{2l'-2})]$, $ w_3 = (j_{2l-1}\,\,t_{2l'-1})(j_{2l}\,\,t_{2l'})$, $ w_4 = (i_1\,\,i_2)(i_3\,\,i_4)\cdots (i_{2(r'-l')-1}\,\,i_{2(r'-l')} )$, and $ w_5 = (s_1\,\,s_3)(s_2\,\,s_4)\cdots (s_{2l''-2}\,\,s_{2l''})$. If $ w \in H $, then we are done. Now, $ \mathcal{N}(w) = 1+(l-1)+(l'-1)+2+ r'-l'+l'' = r $ by equation (\ref{eqn2}). Now, $ wz = (1\,\,3\,\,2\,\,4)w_1w_2w_3w_4w_5$ and $ w(zx) = (1\,\,4\,\,2\,\,3)u_1u_2u_3u_4u_5$, where $ u_1 = (j_1\,\,j_4)(j_2\,\,j_3)\cdots(j_{2l-5}\,\,j_{2l-2})(j_{2l-4}\,\,j_{2l-3})$, $ u_2 = (t_1\,\,t_4)(t_2\,\,t_3)\cdots (t_{2l'-5}\,\,t_{2l'-2})(t_{2l'-4}\,\,t_{2l'-3})$, $ u_3 = (j_{2l-1}\,\,t_{2l'})(j_{2l}\,\,t_{2l'-1}) $, $ u_4 = w_4$, and $ u_5 = (s_1\,\,s_4)(s_2\,\,s_3)\cdots(s_{2l''-3}\,\,s_{2l''})$. Then $ \mathcal{N}(wz) = (l-1)+(l'-1)+2+(r'-l')+l'' = l+r'+l'' = r-1 $ by equation (\ref{eqn3}) and $ \mathcal{N}(wzx) = (l-1)+(l'-1)+2+(r'-l')+l'' = l+r'+l'' = r-1 $ by equation (\ref{eqn3}). Thus, $ cl(wz) = cl(wzx) $ as $ wz$ and $ wzx $ contains exactly one $ 4 $ cycle. Therefore, we are done.

\vspace{0.1cm}

\noindent Suppose $ l $ is even. Then it follows from equation (\ref{eqn1}) that $ l' $ is even. Let us consider the involution $w = (1\,\,2)w_1w_2w_3w_4 $, where $ w_1 = (j_1\,\,j_3)\cdots (j_{2l-2}\,\,j_{2l})$, $ w_2 = (t_1\,\,t_3)\cdots (t_{2l'-2}\,\,t_{2l'})$, $ w_3 = (i_1\,\,i_2)(i_3\,\,i_4)\cdots (i_{2(r'-l')-1}\,\,i_{2(r'-l')} )$, and $w_4 = (s_1\,\,s_3)(s_2\,\,s_4)\cdots (s_{2l''-2}\,\,s_{2l''})$. Then $ wz = (1\,\,3\,\,2\,\,4)w_1w_2w_3w_4$ and $\mathcal{N}(wz) = l+l'+r'-l'+l'' = \mathcal{N}(zy)$. Now, $ w(zx) =(1\,\,4\,\,2\,\,3)u_1u_2u_3$, where $ u_1 = (j_1\,\,j_4)(j_2\,\,j_3)\cdots (j_{2l-2}\,\,j_{2l-1})(j_{2l-3}\,\,j_{2l})$, $ u_2 = (t_1\,\,t_4)(t_2\,\,t_3)\cdots (t_{2l'-2}\,\,t_{2l'-1})(t_{2l'-3}\,\,t_{2l'})$, $ u_3 = w_3$, and $ u_4 = (s_1\,\,s_4)(s_2\,\,s_3)\cdots (s_{2l''-2}\,\,s_{2l''-1})(s_{2l''-3}\,\,s_{2l''}).$ Then $ \mathcal{N}(wzx) = l+l'+(r'-l')+l'' = \mathcal{N}(zy).$ This shows that $ cl(wzx)=cl(zy)=cl(wz)$ as $ wzx,~wz,~zy $ contains exactly one $ 4$ cycle. Therefore, we are done.

\vspace{0.1cm}

\noindent \textbf{Subcase B:} When $ k $ is odd.

\noindent Suppose $ r $ is odd and $ l $ is even. Then it follows from equations (\ref{eqn1}) and (\ref{eqn3}) that $ l' $ is odd and $ r'$ is even.  Now, consider the element $ w = (1\,\,2)(3\,\,4)w_1w_2w_3w_4w_5$, where $ w_1 = [(j_1\,\,j_2)\cdots(j_{2l-1}\,\,j_{2l})]$, $w_2 = [(i_1\,\,i_2)\cdots (i_{2l'-3}\,\,i_{2l'-2})]$, $w_3=[(i_{2l'-1}\,\,i_{2l'}\,\,i_{2l'+1}\,\,i_{2l'+2})]$, $w_4=[(i_{2l'+3}\,\,i_{2l'+5})\cdots (i_{2r'-2}\,\,i_{2r'})]$, and $ w_5 = (s_1\,\,s_2)\cdots(s_{2l''-1}\,\,s_{2l''})$. Now, $ \mathcal{N}(w) = 2+l+(r'-2)+l''= r-1$ by equation (\ref{eqn3}). It is clear that $ cl(wz)= cl(w)=cl(zy).$ Now, $ w(zxy) = (1\,\,4\,\,2\,\,3)u_1u_2u_3u_4u_5$ where $ u_1 = (t_1\,\,t_2)\cdots(t_{2l'-1}\,\,t_{2l'})$, $ u_2 = (j_1\,\,j_2)\cdots(j_{2l-1}\,\,j_{2l})$, $ u_3 = (i_{2l'-1}\,\,i_{2l'+1})$, $ u_4 = (i_{2l'+3}\,\,i_{2l'+6})(i_{2l'+4}\,\,i_{2l'+5})\cdots(i_{2r'-2}\,\,i_{2r'-1})(i_{2r'-3}\,\,i_{2r'})$, and $ u_5 = (s_1\,\,s_3)(s_2\,\,s_4)\cdots (s_{2l''-2}\,\,s_{2l''})$. Then $ \mathcal{N}(wzxy)= l'+l+1+ r'-(l'+1)+l'' = r-1$ by equation (\ref{eqn3}). This shows that $ cl(wz) = cl(wzxy) = cl(zy).$ Hence, we are done.

\vspace{0.1cm}

\noindent Suppose $ l $ is odd and $ r $ is even. Then it follows from the equations (\ref{eqn1}) and (\ref{eqn3}) that $l'$ and $ r'$ are even. It follows from equations (\ref{eqn1}) and (\ref{eqn3}) that $ r'>0$ as $ k\leq r$. Let us consider the element $ w = (1\,\,3\,\,2\,\,i_1)(4\,\,i_2)w_1w_2w_3$, where $ w_1 = (j_1\,\,j_3)\cdots(j_{2l-4}\,\,j_{2l-2})$, $ w_2 = (j_{2l-1}\,\,i_{2r'-1})(j_{2l}\,\,i_{2r'})$, and $ w_3 = (i_3\,\,i_5)\cdots(i_{2r'-4}\,\,i_{2r'-2})$. Then $ \mathcal{N}(w) = 1+l-1+2+r'-2 = r-1 $ by equation (\ref{eqn3}). If $ w \in H $, then we are done. Suppose $ w\notin H$. Then $ wy = (1\,\,i_1\,\,4\,\,i_2)(2\,\,3)u_1u_2u_3 $, where $ u_1 = (j_1\,\,j_4)(j_2\,\,j_3)\cdots(j_{2l-4}\,\,j_{2l-3})(j_{2l-5}\,\,j_{2l-2})$, $ u_2 = (j_{2l-1}\,\,i_{2r'})(j_{2l}\,\,i_{2r'-1})$, and $ u_3 = (i_3\,\,i_6)(i_4\,\,i_5)\cdots(i_{2r'-4}\,\,i_{2r'-3})(i_{2r'-5}\,\,i_{2r'-2})$. This implies that $ cl(wy) = cl(zy) = cl(w)$. Consequently, we are done.

\vspace{0.1cm}

\noindent\textbf{Case IV:} When $  |\Delta(z)\cap \Delta(y)| = 3 .$

\noindent Then there exists a transposition $(a_1\,\,a_2) $ of $ y $ such that $ a_1 \in\Delta(z) $ and $ a_2\notin \Delta(z)$. Without loss of generality, let $ \Delta(z)\cap \Delta(y) = \{1,\,\,2,\,\,3\}$, $ a_1 = 3 $ and $ y = (i_1\,\,i_2)\cdots(i_{2r-1}\,\,i_{2r})$. It implies that $ y(4) = 4 $. Suppose $ xyx^{-1} = y.$ Then $  (x(i_1)\,\,x(i_2))(x(i_3) \,\, x(i_4)) \cdots (x(i_{2r-1})\,\, x(i_{2r})) = y $ due to Lemma \ref{lmn2.9}. It implies that $ x(a_1) = i_{r'}$ for some $ r'\in \{1,2,\dots,2r-1\}$. Consequently, $x(a_2) = i_{r'+1}$. Then $ x(3) = i_{r'} = 4 $ by definition of $ x $. Now, $ y(i_r') = i_{r'+1}$. Then $ y(4) = i_{r'+1}.$ Consequently, $ 4 = i_{r'+1} = i_r' $ as $ y(4) = 4 $. Thus, we arrive a contradiction as $ i_r'\neq i_{r'+1} $. Hence, $ xyx \neq y $. Therefore, $ zy \neq z(xyx)$. Let $zy\in cl(g) $ for some $ g\in S_n $. Then $ x(zy)x^{-1} \in cl(g)$. Since $ xz = zx $, therefore, $ z(xyx) \in cl(g)$. Consequently, $ cl(zy) = cl(zxyx).$ Therefore, we are done.

\vspace{0.1cm}

\noindent\textbf{Case V:} When $ |\Delta(z)\cap \Delta(y)| = 4.$

\noindent Suppose $ y $ contains one of elements $(1\,\,4)(2\,\,3)$, $  (1\,\,2)(3\,\,4)$. Then $ cl(zy) = cl(y)$. Consequently, we are done. 

\noindent Suppose $ z $ acts nontrivially atleast three transpositions of $ y $. Then $ xy\neq yx $ due to Lemma \ref{lmn2.10}. Consequently, $ zy $ and $ zxyx $ are distinct elements in $ zH$ with $ cl(zy) = cl(zxyx)$ as $ xz = zx $. Therefore, we are done. 

\vspace{0.1cm}

\noindent Suppose $ x $ and $ y $ share $ l $ transpositions. 
\noindent Without loss of generality, let $ y = (1\,\,3)(2\,\,4)y_1y_2y_3$, where $ y_1 = (j_1\,\,j_2)\cdots(j_{2l-1}\,\,j_{2l})$, $y_3=(s_1\,\,s_3)\cdots(s_{2l''-2}\,\,s_{2l''})$, and $y_3 = (i_1\,\,i_2)\cdots(i_{2r'-1}\,\,i_{2r'})$. Here, \begin{equation}\label{eqn4}
    2+l+l''+r' = r.
\end{equation} Then $ zy = (j_1\,\,j_2)\cdots(j_{2l-1}\,\,j_{2l})(s_1\,\,s_3)\cdots(s_{2l''-2}\,\,s_{2l''})(i_1\,\,i_2)\cdots(i_{2r'-1}\,\,i_{2r'})$ and $\mathcal{N}(zy) = r-2 $ by equation (\ref{eqn4}). If $ r $ is even, then $ zy $ is square element. Consequently, we are done. 

\vspace{0.1cm}

\noindent Suppose $ r $ is odd.
To proceed further, we consider the following sub-cases.

\vspace{0.1cm}

\noindent\textbf{Subcase A:} When $ k $ is odd. 

\noindent Let us the consider the involution $ w = (1\,\,3)(2\,\,4)w_1w_2w_3w_4$, where $ w_1 = (j_1\,\,j_2)\cdots(j_{2l-1}\,\,j_{2l})$, $w_2 = (t_1\,\,t_2)\cdots (t_{2l'-1}\,\,t_{2l'})$, $w_3 = (i_1\,\,i_2)(i_3\,\,i_4)[(i_5\,\,i_7)\cdots(i_{2(r'-l')-2}\,\, i_{2(r'-l')})]$, and $ w_4 = (s_1\,\,s_2)\cdots(s_{2l''-1}\,\,s_{2l''})$. Then $ wz = w_1w_2w_3w_4$. Now, $\mathcal{N}(wz) = l+l'+(r'-l')+l'' = \mathcal{N}(zy)$. Consequently, $ cl(wz) = cl(zy).$ Now, $ w(zxy) = (1\,\,4)(2\,\,3) u_1u_2u_3 u_4$, where $ u_1 = [(j_1\,\,j_2)\cdots(j_{2l-1}\,\,j_{2l})]$, $ u_2 = [(i_5\,\,i_8)(i_6\,\,i_7)\cdots (i_{2(r'-l')-3}\,\,i_{2(r'-l')})]$, $ u_3 = [(i_{2(r'-l')+1}\,\,i_{2(r'-l')+2})\cdots(i_{2r'-1}\,\,i_{2r'})]$, and $ u_4 = (s_1\,\,s_4)(s_2\,\,s_3)\cdots (s_{2l''-2}\,\,s_{2l''-1})(s_{2l''-3}\,\,s_{2l''})$. Then $ \mathcal{N}(wzxy) = 2+l +r'-2+l'' = \mathcal{N}(zy) = \mathcal{N}(wz).$ Consequently, $ cl(wzxy) = cl(zy) = cl(wz)$. Therefore, we are done.   

\vspace{0.1cm}

\noindent\textbf{Subcase B:} When $ k $ is even.

\noindent Suppose $ l $ is odd. Then it follows from equations $(\ref{eqn1})$ and $ (\ref{eqn4})$ that $ l '$ is odd and $ r' $ is even. Let us consider the involution $ w = (1\,\,3)(2\,\,4)w_1w_2w_3w_4$, where $ w_1 = [(j_1\,\,j_3)(j_2\,\,j_4)\cdots (j_{2l-5}\,\,j_{2l-3})(j_{2l-4}\,\,j_{2l-2})](j_{2l-1}\,\,j_{2l})$, $w_2 = [(t_1\,\,t_3)\cdots (t_{2l'-4}\,\,t_{2l'-2})](t_{2l'-1}\,\,t_{2l'})$, $w_3 = [(i_1\,\,i_2)(i_3\,\,i_4)\cdots(i_{2(r'-l')-1}\,\,i_{2(r'-l')})]$, and $ w_4 = [(s_1\,\,s_3)(s_2\,\,s_4)\cdots (s_{2l''-2}\,\,s_{2l''})]$. Then $ wz = w_1w_2w_3w_4$ and $ \mathcal{N}(wz) = l+l'+r'-l' +l''= \mathcal{N}(zy)$. Consequently, $ cl(wz) = cl(zy)$. Now, $ w(zx) = (1\,\,2)(3\,\,4)u_1u_2u_3u_4$, where $ u_1 = [(j_1\,\,j_4)(j_2\,\,j_3)\cdots (j_{2l-5}\,\,j_{2l-2})]$, $ u_2 = [(t_1\,\,t_4)(t_2\,\,t_3)\cdots (t_{2l'-5}\,\,t_{2l'-2})]$, $ u_3 = w_3$, and $u_4 = (s_1\,\,s_4)(s_2\,\,s_3)\cdots(s_{2l''-2}\,\,s_{2l''-1})(s_{2l''-3}\,\,s_{2l''})$. Then $ \mathcal{N}(wzx) = 2+l-1+l'-1+r'-l'+l'' = \mathcal{N}(zy)$. Consequently, $ cl(wzx) = cl(zy) = cl(wz).$ Therefore, we are done.

\vspace{0.1cm}

\noindent Suppose $ l $ is even. Then it follows from equations $(\ref{eqn1})$ and $ (\ref{eqn4})$ that $ l '$ is even and $ r' $ is odd. Let us consider the involution $ w = (1\,\,3)(2\,\,4)w_1w_2w_3w_4$, where $ w_1 = (j_1\,\,j_2)(j_3\,\,j_4)[(j_5\,\,j_7)(j_6\,\,j_8)\cdots (j_{2l-3}\,\,j_{2l-1})(j_{2l-2}\,\,j_{2l})]$, $w_2 = [(t_1\,\,t_3)\cdots (t_{2l'-2}\,\,t_{2l'})]$, $w_3 = [(i_1\,\,i_2)(i_3\,\,i_4)\cdots(i_{2(r'-l')-1}\,\,i_{2(r'-l')})]$, and $ w_4 = (s_1\,\,s_3)(s_2\,\,s_4)\cdots (s_{2l''-3}\,\,s_{2l''-1})(s_{2l''-2}\,\,s_{2l''})$. Then $ wz = w_1w_2w_3w_4$ and $ \mathcal{N}(wz) = l+l'+r'-l'+l''= \mathcal{N}(zy)$. Consequently, $ cl(wz) = cl(zy)$. Now, $ w(zx) = (1\,\,2)(3\,\,4)u_1u_2u_3u_4$, where $ u_1 = [(j_5\,\,j_8)(j_6\,\,j_7)\cdots (j_{2l-2}\,\,j_{2l-1})(j_{2l-3}\,\,j_{2l})]$, $ u_2 = [(t_1\,\,t_4)(t_2\,\,t_3)\cdots (t_{2l'-3}\,\,t_{2l'})]$, $ u_3 = w_3$, and $ u_4 = (s_1\,\,s_4)(s_2\,\,s_3)\cdots(s_{2l''-2}\,\,s_{2l''})$. Then $ \mathcal{N}(wzx) = 2+l-2+l'+r'-l'+l'' = \mathcal{N}(zy)$. Consequently, $ cl(wzx) = cl(zy) = cl(wz)$ and this concludes the case.

\vspace{0.1cm}

\noindent It follows from the above cases that our claim is justified. Consequently, $H$ is not a perfect subgroup code of $S_n$ in the Cayley sum graph due to Lemma \ref{lmn3.2}, and this completes the proof.
\end{proof}
As a consequence of the aforesaid result, we present the following corollary whose proof is omitted as it is akin to the proof of Theorem \ref{thm5}.
\begin{cor}\label{col6}
Let $ H $ be a nontrivial even order subgroup of $ A_n $ such that $ I_{min}(H) > 2$. Then $ H $ is not a perfect subgroup code of $ A_n $ in the Cayley sum graph.
\end{cor}

It is easy to see that the trivial subgroup $\{1_{S_n}\}$ is not a perfect subgroup code of $ S_n $ as well as $ A_n $ in the Cayley sum graph. We now present the main theorems of this section, which follow from the preparatory work in the preceding results. Our principal result completely characterize perfect subgroup codes of $ S_n $ and $ A_n $  which concludes this section.
\begin{theorem}
The only perfect subgroup code of $S_n$ (or $A_n$) in the Cayley sum graph is $S_n$ (or $ A_n$) itself. 
\end{theorem}

%%%%%%%%%%%%%%%%%%%%%%%%%%%%%%%%%%%%%%%%%%%%%


\begin{thebibliography}{99}

\bibitem{Amo} Amooshahi M., Taeri B.,  \textit{On Cayley sum graphs of non-abelian groups}. Graphs Comb. \textbf{32}, 17–29 (2016).

\bibitem{Biggs1973} Biggs N.L., \textit{Perfect codes in graphs}. J. Comb. Theory Ser. B \textbf{15}, 289–296 (1973).

%\bibitem{Brouwer1989} Brouwer A.E., Cohen A.M., Neumaier A., \textit{Distance-regular Graphs}. Springer, Berlin (1989)

\bibitem{Chen2020} Chen J., Wang Y., Xia B., \textit{Characterization of subgroup perfect codes in Cayley graphs}. Discret. Math. \textbf{343}, 111813 (2020).

\bibitem{Chung} Chung F.R.K.,  \textit{Diameters and eigenvalues}. J. Am. Math. Soc. \textbf{2}, 187–196 (1989).

\bibitem{Chihara} Chihara L., \textit{On the zeros of the Askey–Wilson polynomials, with applications to coding theory}. SIAM J. Math. Anal. \textbf{18}(1), 191–207 (1987).

\bibitem{Che} Cheyne B., Gupta V., Wheeler C., \textit{Hamilton cycles in addition graphs}. Rose-Hulman Undergrad. Math J. \textbf{4}, 1–17 (2003). (electronic)

\bibitem{Cornad} Conrad K., \textit{Generating sets}. University of Connecticut (2025). https://kconrad.math.uconn.edu/blurbs/grouptheory/genset.pdf (Accessed on 04.07.2025).

\bibitem{Dejter2003} Dejter I.J., Serra O., \textit{Efficient dominating sets in Cayley graphs}. Discret. Appl. Math. \textbf{129}, 319–328(2003).

\bibitem{ev} DeVos M., Goddyn L., Mohar B., Sámal R., \textit{Cayley sum graphs and eigenvalues of (3, 6)-fullerenes}. J.
Combin. Theory Ser. B \textbf{99}, 358–369 (2009).

\bibitem{Dixon} Dixon J.D., Mortimer B., \textit{Permutation Groups}. Springer, (1996). 

%\bibitem{Feng2017} Feng R., Huang H., Zhou S., \textit{Perfect codes in circulant graphs}. Discret. Math. \textbf{340}, 1522–1527 (2017).

\bibitem{Gry} Grynkiewicz D., Levb V.F., Serra O., \textit{Connectivity of addition Cayley graphs}. J. Combin. Theory Ser. B. \textbf{99}, 202–217 (2009).

%\bibitem{Golomb1970} Golomb S.W., Welch L.R., \textit{Perfect codes in the Lee metric and the packing of polyominoes}. SIAM J. Appl. Math. \textbf{18}, 302–317 (1970).

%\bibitem{Hammond} Hammond P., Smith D.H., \textit{Perfect codes in the graphs $O_k$}. J. Comb. Theory Ser. B \textbf{19}, 239–255 (1975).


\bibitem{Horak} Horak P., Kim D., \textit{ 50 years of the Golomb–Welch conjecture}. IEEE Trans. Inform. Theory \textbf{64}, 3048–3061 (2018).

\bibitem{Huang2018} Huang H., Xia B., Zhou S., \textit{Perfect codes in Cayley graphs}. SIAM J. Discret. Math. \textbf{32}, 548–559 (2018).

\bibitem{Kon} Konyagin S.V., Shkredov I.D., \textit{On subgraphs of random Cayley sum graphs}. Eur. J. Comb. \textbf{70}, 61–74 (2018).

\bibitem{Krotov2020} Krotov D.S., \textit{The existence of perfect codes in Doob graphs}. IEEE Trans. Inf. Theory \textbf{66}(3), 1423–1427 (2020).


\bibitem{Lev} Lev V.F., \textit{Sums and differences along Hamiltonian cycles}. Discret. Math. \textbf{310}, 575–584 (2010).

\bibitem{Lee2001} Lee J., \textit{Independent perfect domination sets in Cayley graphs}. J. Graph Theory \textbf{37}, 213–219 (2001).

\bibitem{Lloyd}Lloyd S.P., \textit{Binary block coding}. Bell Syst. Tech. J. \textbf{36}, 517–535 (1957).

\bibitem{Ma2016} Ma X., Wang K., \textit{Integral Cayley sum graphs and groups}. Discuss. Math. Graph Theory. \textbf{36}, 797–803 (2016).



\bibitem{Ma} Ma X., Walls G.L., Wang K., Zhou S., \textit{Subgroup perfect codes in Cayley graphs}. SIAM J. Discret. Math. \textbf{34}, 1909–1921 (2020).

\bibitem{MacWilliams1977} MacWilliams F.J., Sloane N.J.A., \textit{The Theory of Error-Correcting Codes}. North-Holland, Amsterdam (1977).

\bibitem{malik97} Malik D.S., Mordeson J.N., Sen M.K., \textit{Fundamentals of Abstract Algebra}. The McGraw-Hill Companies, Inc. New York (1997).

%\bibitem{Martin} Martin W.J., Zhu X.J., \textit{Anticodes for the Grassman and bilinear forms graphs}. Des. Codes Cryptogr. \textbf{6}(1), 73–79 (1995).

\bibitem{Ma2023} Ma X., Wang K., Yang Y., \textit{Perfect codes in Cayley sum graphs}.  Electronic journal of combinatorics. \textbf{29}, (2022).

\bibitem{Ma2020} Ma, X., Feng, M., Wang, K. \textit{Subgroup perfect codes in Cayley sum graphs}. Des. Codes Cryptogr. \textbf{88}, 1447–1461 (2020).


\bibitem{Suzuki19821} Suzuki M., \textit{Group Theory I}. Springer, New York (1982).


\bibitem{Zhang2021} Zhang J., Zhou S., \textit{On subgroup perfect codes in Cayley graphs}. Eur. J. Comb. \textbf{91}, 103228 (2021).

\bibitem{Zhang2022} Zhang J., Zhou S., \textit{Corrigendum to “On subgroup perfect codes in Cayley graphs}. [Eur. J. Comb. 91, 103228 (2022)]”. Eur. J. Comb. \textbf{101}, 103461 (2022).

\bibitem{Zhang2023} Zhang J., \textit{Characterizing subgroup perfect codes by 2-subgroups}.  Des. Codes Cryptogr. \textbf{91},) 2811–2819 (2023).

\bibitem{Zhang24} Zhang J., Zhu Y., \textit{A note on regular sets in Cayley graphs}. Bulletin of the Australian Mathematical Society \textbf{109}, 1-5 (2024).

\bibitem{Zhang2024} Zhang J., \textit{On subgroup perfect codes in Cayley sum graphs}. Finite Fields and Their Applications. \textbf{95}, 102393 (2024).

\end{thebibliography}
\end{document}